\documentclass[12pt]{amsart}
\usepackage[margin=1.1in]{geometry}

\usepackage{amsmath, amssymb, mathtools}
\usepackage{mathrsfs}
\usepackage{graphicx}
\usepackage{tikz}
\usepackage{amsthm}
\usepackage{hyperref}
\usepackage{constants}
\newconstantfamily{abcon}{symbol=c}
\numberwithin{equation}{section}
\usepackage{enumitem}

\newcommand{\Z}{\mathbb{Z}}
\newcommand{\Q}{\mathbb{Q}}
\newcommand{\R}{\mathbb{R}}

\newcommand{\OO}{\mathcal{O}}

\newcommand\Vol{\mathrm{Vol}}

\newcommand\Gal{\mathrm{Gal}}
\newcommand\Norm{\mathrm{N}}

\renewcommand\epsilon{\varepsilon}

\newtheorem{lemma}{Lemma}[section]
\newtheorem{theorem}[lemma]{Theorem}
\newtheorem{proposition}[lemma]{Proposition}
\newtheorem{corollary}[lemma]{Corollary}
\newtheorem{mydef}[lemma]{Definition}

\newtheorem*{remark}{{\it Remark}}

\title{Bounds for moments of $\ell$-torsion in class groups}

\author{Peter Koymans}
\address{Institute for Theoretical Studies, ETH Zurich, 8092 Zurich, Switzerland} 
\email{\href{mailto:peter.koymans@eth-its.ethz.ch}{peter.koymans@eth-its.ethz.ch}}

\author{Jesse Thorner}
\address{Department of Mathematics, University of Illinois, Urbana, IL 61801, USA}
\email{\href{mailto:jesse.thorner@gmail.com}{jesse.thorner@gmail.com}}

\begin{document}

\maketitle

\begin{abstract}
Fix a number field $k$, integers $\ell, n \geq 2$, and a prime $p$.  For all $r \geq 1$, we prove strong unconditional upper bounds on the $r$-th moment of $\ell$-torsion in the ideal class groups of degree $p$ extensions of $k$ and of degree $n$ $S_n$-extensions of $k$, improving upon results of Ellenberg, Pierce and Wood as well as GRH-conditional results of Frei and Widmer.  For large $r$, our results are comparable with work of Heath-Brown and Pierce for imaginary quadratic extensions of $\mathbb{Q}$.  When $r=1$, our results are new even for the family of all quadratic extensions of $\mathbb{Q}$, leading to an improved upper bound for the count of degree $p$ $D_p$-extensions over $\mathbb{Q}$ (where $D_p$ is the dihedral group of order $2p$).
\end{abstract}

\section{Introduction and statement of the main results}
Cohen, Lenstra, and Martinet \cite{CL,CM} gave heuristics that predict the distribution of ideal class groups $\mathrm{Cl}_{K}$ of number fields $K$ in certain families, including the distribution of their $\ell$-torsion subgroups $\mathrm{Cl}_{K}[\ell]$ for certain primes $\ell$.  Fix an algebraic closure $\overline{\Q}$. Let $d\geq 2$ be an integer, $G$ be a transitive group of the symmetric group $S_d$, $\ell\nmid|G|$ be a ``good'' prime, $K/k$ be a finite extension of number fields chosen inside $\overline{\mathbb{Q}}$, and $D_K$ be the absolute discriminant of $K$. Let $\widetilde{K}$ be the Galois closure of $K$ over $k$ inside $\overline{\Q}$. Define the families
\begin{equation}
\label{eqn:families}
\mathscr{F}_{k}^{d,G}:=\{K\colon \textup{$[K:k]=d$,}~\Gal(\widetilde{K}/k)\cong G\},\qquad \mathscr{F}_{k}^d:=\{K\colon [K:k]=d\}.
\end{equation}
For a subset $\mathcal{S}\subseteq \mathscr{F}_k^{d,G}$ or $\mathcal{S}\subseteq \mathscr{F}_k^{n}$, we define $\mathcal{S}(Q):=\{K\in\mathcal{S}\colon D_K\in(Q,2Q]\}$ and
\begin{equation}
\label{eqn:families2}
\alpha_{\mathcal{S}}:=\liminf_{Q\to\infty}\frac{\log|\mathcal{S}(Q)|}{\log Q},\qquad \beta_{\mathcal{S}}:=\limsup_{Q\to\infty}\frac{\log|\mathcal{S}(Q)|}{\log Q}.
\end{equation}

It is conjectured that there exists a constant $c_{G,d,k,\ell}>0$ such that
\begin{equation}
\label{eqn:CLM}
\lim_{Q\to\infty}\frac{1}{|\mathscr{F}_{k}^{d,G}(Q)|}\sum_{K\in \mathscr{F}_{k}^{d,G}(Q)}|\mathrm{Cl}_{K}[\ell]| = c_{G,d,k,\ell}.
\end{equation}
So far, the existence of the limit \eqref{eqn:CLM} is only known when $G=S_2$ and $\ell=3$ (\cite{DavenportHeilbronn} when $k=\Q$, \cite{DatskovskyWright} otherwise), $G=S_3$ and $\ell=2$ (\cite{Bhargava5} when $k=\Q$, \cite{BSW_geo} otherwise), or $G\subseteq S_{2^m}$ is a transitive permutation 2-group containing a transposition and $\ell=3$ \cite{LemkeOliverWangWood}.  Since we cannot establish \eqref{eqn:CLM} in full generality yet, we instead try to bound the moments
\begin{equation}
\label{eqn:moments_1}
\sum_{K\in \mathcal{S}(Q)} |\mathrm{Cl}_{K}[\ell]|^r, \qquad \ell,r \geq 1.	
\end{equation}
Duke \cite{Duke98} conjectured that if $F/E$ is any extension of number fields and $\epsilon>0$, then $|\mathrm{Cl}_{F}[\ell] |\ll_{[F:E], E, \ell, \epsilon} D_F^{\epsilon}$.  This would imply that \eqref{eqn:moments_1} is $O_{d,k,\ell,r,\epsilon}(Q^{\epsilon}|\mathcal{S}(Q)|)$.    Extending work of Gau{\ss}, Kl\"uners and Wang \cite{KlunersWang} proved Duke's conjectured bound for $|\mathrm{Cl}_{F}[\ell]|$ when $F/E$ is a Galois extension whose Galois group $\Gal(F/E)$ is an $\ell$-group. As of now, for $K \in \mathscr{F}_{k}^{d}$, the only uniform pointwise bound for $|\mathrm{Cl}_{K}[\ell]|$ is the trivial bound
\begin{equation}
\label{eqn:trivial_minkowski}
|\mathrm{Cl}_{K}[\ell] |\leq |\mathrm{Cl}_{K}| \ll_{d,[k:\Q]} D_K^{\frac{1}{2}}(\log D_K)^{[K:\Q]-1}\ll_{d,[k:\Q],\epsilon}D_K^{\frac{1}{2}+\epsilon}
\end{equation}
due to Landau.  Under the generalized Riemann hypothesis (GRH) for the Dedekind zeta function $\zeta_{\widetilde{K}}(s)$, Ellenberg and Venkatesh \cite{EV} proved that if $\epsilon>0$, then
\begin{equation}
\label{eqn:GRH}
|\mathrm{Cl}_{K}[\ell]| \ll_{d, k, \ell, \epsilon} D_K^{\frac{1}{2} - \frac{1}{2\ell([K:k] - 1)} + \epsilon}.
\end{equation}
Currently, unconditional power-saving improvements over \eqref{eqn:trivial_minkowski} are only known in special cases and are quite hard to prove. In addition to \cite{KlunersWang}, see \cite{BSTTTZ, EV, HV, MR2254390, PTW2, Wang-Nilpotent, Wang}.

Heath-Brown and Pierce \cite{HBP} proved several strong bounds on moments of $|\mathrm{Cl}_{K}[\ell]|$ as $K\in\mathscr{F}_{\Q}^{2}(Q)$ runs through the imaginary quadratic extensions of $\Q$.  We present two particularly strong and elegant bounds \cite[Theorems 1.1 and 1.5]{HBP}:  If $\ell\geq 5$, then
\begin{align}
\label{eqn:HBP_1}
&\sum_{\substack{K\in\mathscr{F}_{\Q}^2(Q) \\ \textup{$K/\Q$ imaginary}}} |\mathrm{Cl}_{K}[\ell]| \ll_{\ell, \epsilon} Q^{\frac{3}{2}(1-\frac{1}{\ell+1}) + \epsilon},\\
\label{eqn:HBP}
&\sum_{\substack{K\in\mathscr{F}_{\Q}^2(Q) \\ \textup{$K/\Q$ imaginary}}} |\mathrm{Cl}_{K}[\ell]|^{r} \ll_{\ell, r, \epsilon} Q^{\frac{r}{2} + \epsilon},\quad r \geq \ell+1.
\end{align}
Since $|\{K\in\mathscr{F}_{\Q}^2(Q)\colon \textup{$K/\Q$ imaginary}\}|\asymp Q$, both \eqref{eqn:HBP_1} and \eqref{eqn:HBP} show that on average, one obtains a power-saving improvement over \eqref{eqn:GRH}.  Also, \eqref{eqn:HBP} recovers the trivial bound \eqref{eqn:trivial_minkowski} when all but one of the terms on the left hand side are dropped.

The work of Ellenberg, Pierce, and Wood \cite[Corollaries 1.1.1 and 1.1.2]{EPW} addresses \eqref{eqn:moments_1} when $r=1$ for number fields $K/\Q$ with $3 \leq [K:\Q] \leq 5$.  Their work also provides a variant of \eqref{eqn:HBP_1} with an average over {\it all} $K\in\mathscr{F}_{\Q}^2(Q)$.  They proved that if $\epsilon > 0$ and $\ell\geq 2$, then
\begin{equation}
\label{eqn:EPW1}
\begin{aligned}
  \begin{split}
\sum_{K\in \mathscr{F}_{\Q}^2(Q)}|\mathrm{Cl}_{K}[\ell]|&\ll_{\ell,\epsilon} Q^{\frac{3}{2}-\frac{1}{2\ell}+\epsilon},\\
\sum_{K\in \mathscr{F}_{\Q}^{4,S_4}(Q)}|\mathrm{Cl}_{K}[\ell]|&\ll_{\ell,\epsilon} Q^{\frac{3}{2}-\min\{\frac{1}{48},\frac{1}{6\ell}\}+\epsilon},
  \end{split}
\qquad
  \begin{split}
\sum_{K\in \mathscr{F}_{\Q}^3(Q)}|\mathrm{Cl}_{K}[\ell]|&\ll_{\ell,\epsilon} Q^{\frac{3}{2}-\frac{1}{4\ell}+\epsilon},\\
\sum_{K\in \mathscr{F}_{\Q}^5(Q)}|\mathrm{Cl}_{K}[\ell]|&\ll_{\ell,\epsilon} Q^{\frac{3}{2}-\min\{\frac{1}{200},\frac{1}{8\ell}\}+\epsilon}.
  \end{split}
\end{aligned}
\end{equation}
These are proved by establishing the GRH-conditional bound \eqref{eqn:GRH} on average.  Note that $|\mathscr{F}_{\Q}^2(Q)|$, $|\mathscr{F}_{\Q}^3(Q)|$, $|\mathscr{F}_{\Q}^{4,S_4}(Q)|$, and $|\mathscr{F}_{\Q}^5(Q)|$ are each $\asymp Q$.

Building on the ideas in \cite{EV} and their work in \cite{FW2,Widmer}, Frei and Widmer improved upon \eqref{eqn:EPW1}.  In some cases, their method can address higher moments.  As a sampling of their numerous unconditional results in \cite{FW}, they proved that if $\ell,r\geq 1$, then
\begin{equation}
\label{eqn:EPW3}
\begin{aligned}
  \begin{split}
\sum_{K\in \mathscr{F}_{\Q}^2(Q)}|\mathrm{Cl}_{K}[\ell]|^r&\ll_{\ell,\epsilon} Q^{\frac{r}{2}+1-\min\{1,\frac{r}{\ell+2}\}+\epsilon},\\
\sum_{K\in \mathscr{F}_{\Q}^{4,S_4}(Q)}|\mathrm{Cl}_{K}[\ell]|&\ll_{\ell,\epsilon} Q^{\frac{3}{2}-\min\{\frac{1}{48},\frac{1}{3\ell+3}\}+\epsilon},
  \end{split}
\quad
  \begin{split}
\sum_{K\in \mathscr{F}_{\Q}^3(Q)}|\mathrm{Cl}_{K}[\ell]|&\ll_{\ell,\epsilon} Q^{\frac{3}{2}-\min\{\frac{2}{25},\frac{1}{2\ell+3}\}+\epsilon},\\
\sum_{K\in \mathscr{F}_{\Q}^5(Q)}|\mathrm{Cl}_{K}[\ell]|&\ll_{\ell,\epsilon} Q^{\frac{3}{2}-\min\{\frac{1}{200},\frac{1}{4\ell+3}\}+\epsilon}.
  \end{split}
\end{aligned}\hspace{-3.57mm}
\end{equation}
If $j\in\{1,2,3,4\}$ and $(\ell_1,\ell_2,\ell_3,\ell_4)=(3,4,9,26)$, then the $j$-th line of \eqref{eqn:EPW3} improves upon the $j$-th line of \eqref{eqn:EPW1} and provides an on-average power-saving improvement over \eqref{eqn:GRH} when $\ell\geq \ell_j$.  Unlike \eqref{eqn:HBP}, the second, third, and fourth lines of \eqref{eqn:EPW3} do not recover the trivial bound \eqref{eqn:trivial_minkowski}.  Since $|\mathscr{F}_{\Q}^2(Q)|\asymp Q$, the first line recovers \eqref{eqn:trivial_minkowski} when $r\geq \ell+2$, a range that is weaker than the result in \eqref{eqn:HBP} for {\it imaginary} quadratic extensions of $\Q$.
 
Let $d\geq 2$, let $\mathcal{S}\subseteq\mathscr{F}_{\Q}^d$, and assume $\beta_{\mathcal{S}}>0$.  Note that $\beta_{\mathcal{S}}$ is finite \cite{LemkeOliverThorne}.  Frei and Widmer \cite[Theorem 1.4]{FW} proved that if $\zeta_{\widetilde{K}}(s)$ satisfies GRH for each $K\in \mathcal{S}$, then
\begin{equation}
\label{eqn:FM_2}
\sum_{K\in \mathcal{S}(Q)} |\mathrm{Cl}_{K}[\ell]|^r \ll_{\beta_{\mathcal{S}},d,\ell,r,\epsilon} Q^{\frac{r}{2} + \beta_{\mathcal{S}} - \min\big\{\beta_{\mathcal{S}},\beta_{\mathcal{S}}\frac{r}{\ell(d - 1) + 2}\big\} + \epsilon},\qquad \ell,r\geq 1.
\end{equation}

Fix an integer $n\geq 2$ and a prime $p$.  Recall \eqref{eqn:families} and \eqref{eqn:families2}.  In this paper, we synthesize the ideas of Heath-Brown and Pierce \cite{HBP} with those of Frei and Widmer \cite{FW} to refine the process of estimating moments of $\ell$-torsion in ideal class groups of number fields in any subset $\mathcal{S}$ of the families $\mathscr{F}_{\Q}^p$ and $\mathscr{F}_{\Q}^{n,S_n}$.  In these families, our results uniformly improve upon \eqref{eqn:EPW1}, \eqref{eqn:EPW3}, and \eqref{eqn:FM_2} for all $\ell\geq 2$ and all $r\geq 1$ without any  assumptions on GRH or the sizes of $\alpha_{\mathcal{S}}$ and $\beta_{\mathcal{S}}$.  We do not need to assume GRH because of the recent work of Lemke Oliver, Thorner, and Zaman \cite{LOTZ} on the holomorphy and nonvanishing of Artin $L$-functions (see also the related work of Pierce, Turnage-Butterbaugh, and Wood \cite{PTBW} and Thorner and Zaman \cite{TZ}).  This improvement also holds for $\mathscr{F}_{k}^p$ and $\mathscr{F}_{k}^{n,S_n}$ when $k\neq \Q$ without loss in quality.  We prove the following theorem.

\begin{theorem}
\label{thm:main}
Fix a number field $k$, integers $\ell,n\geq 2$, a prime $p$, and $r\geq 1$.  Let $Q\geq 1$ and $\epsilon>0$. Let the ordered pair $(\mathscr{F}_{k},d)$ equal $(\mathscr{F}_{k}^{p}, p)$ or $(\mathscr{F}_{k}^{n,S_n}, n)$.  If $\mathcal{S}\subseteq\mathscr{F}_k$, then
\[
\sum_{K \in \mathcal{S}(Q)} |\mathrm{Cl}_{K}[\ell]|^r \ll_{\alpha_{\mathcal{S}},d,k,\ell,r,\epsilon} Q^{\frac{r}{2} + \epsilon}\big(1+|\mathcal{S}(Q)|^{1-\frac{r}{\ell(d - 1) + 1}}\big).
\]
\end{theorem}
\begin{remark}
\textup{The implied constant is effectively computable.  We restrict to the families $\mathscr{F}_k^p$ and $\mathscr{F}_k^{n,S_n}$ only so that we can apply the results in \cite{LOTZ} (see Theorem \ref{thm:LOTZ} below).}
\end{remark}

If the ordered pair $(\mathscr{F}_{k},d)$ equals $(\mathscr{F}_{k}^{p}, p)$ or $(\mathscr{F}_{k}^{n,S_n}, n)$, then Theorem \ref{thm:main} implies that
\begin{equation}
\label{eqn:moment2}
\sum_{K \in \mathscr{F}_{k}(Q)} |\mathrm{Cl}_{K}[\ell]|^r \ll_{d,k,\ell,r,\epsilon} Q^{\frac{r}{2} + \epsilon},\qquad r \geq \ell(d-1) + 1.
\end{equation}
This has several appealing features beyond the fact that it is unconditional. First, \eqref{eqn:moment2} recovers \eqref{eqn:trivial_minkowski}. Second, \eqref{eqn:moment2} is completely independent of $|\mathscr{F}_{k}(Q)|$. This is important because of the inexactitude of the existing bounds for $|\mathscr{F}_k(Q)|$.  A weak form of Malle's conjecture \cite{Malle} asserts that $\alpha_{\mathscr{F}_k}=\beta_{\mathscr{F}_k}=1$ in \eqref{eqn:families2}.  As of now, we know that there exists an absolute and effectively computable constant $C>0$ such that
{\small\begin{equation}
\label{eqn:EV_count}
\begin{cases}
1&\mbox{if $d\leq 5$ and $k=\Q$,}\\
\frac{1}{2}+\frac{1}{d-1}&\mbox{if $d\geq 6$ and $k=\Q$,}\\
\frac{1}{2}+\frac{1}{d^2}&\mbox{if $k\neq\Q$}
\end{cases}
\leq \alpha_{\mathscr{F}_k}\leq \beta_{\mathscr{F}_k}\leq \begin{cases}
1&\mbox{if $d\leq 5$ and $k=\Q$,}\\
1.564(\log d)^2&\mbox{if $d\geq 6$ and $k=\Q$,}\\
\exp(C\sqrt{\log d})&\mbox{if $k\neq\Q$.}
\end{cases}
\end{equation}}%
(See \cite{Bhargava4,Bhargava5,BhargavaShankarWangII,DavenportHeilbronn,EV_count,LemkeOliverThorne}.)  Third, if $r=\ell(d-1)+1$, then \eqref{eqn:moment2} implies that on average over $K\in\mathscr{F}_{k}(Q)$, we have a power-saving improvement over \eqref{eqn:GRH} when $\alpha_{\mathscr{F}_k} > \frac{1}{2}+\frac{1}{2\ell(d-1)}$.  Therefore, in this setting, the more accurate {\it lower} bounds in \eqref{eqn:EV_count} matter more than the upper bounds in \eqref{eqn:EV_count}.  Fourth, when $d = 2$, \eqref{eqn:moment2} reduces to \eqref{eqn:HBP} with two new benefits:  (i) when $k=\Q$, the real quadratic extensions of $\Q$ is now included, and (ii) we no longer require that $k=\Q$.  Therefore, Theorem \ref{thm:main} extends \eqref{eqn:HBP} to $\mathscr{F}_{k}$.

Theorem \ref{thm:main} is also new for low moments, producing new results even for quadratic extensions of $\Q$.  In light of \eqref{eqn:EV_count}, Theorem \ref{thm:main} implies that
\begin{equation}
\label{eqn:moment3}
\sum_{K \in \mathscr{F}_{\Q}^{2}(Q)} |\mathrm{Cl}_{K}[\ell]| \ll_{\ell, \epsilon} Q^{\frac{3}{2} -\frac{1}{\ell + 1}+\epsilon}.
\end{equation}
This improves upon the work of Frei and Widmer in the first line of \eqref{eqn:EPW3} when $r=1$, though it does not improve upon \eqref{eqn:HBP_1} for {\it imaginary} quadratic extensions of $\Q$.  We briefly describe an application of \eqref{eqn:moment3}.  For an odd prime $p$, let $D_p$ be the dihedral group of order $2p$, and $D_p(2p)$ be the regular permutation representation of $D_p$.  A conjecture of Malle \cite{Malle} asserts that there exist constants $c(p)>0$ and $c(2p)>0$ such that
\[
\mathscr{F}_{\Q}^{p,D_p}(Q)\sim c(p) Q^{\frac{2}{p-1}}\qquad\textup{and}\qquad \mathscr{F}_{\Q}^{2p,D_p(2p)}(Q)\sim c(2p)Q^{\frac{1}{p}}
\]
as $Q\to\infty$.  Kl{\"u}ners \cite{Kluners} related upper bounds for $\mathscr{F}_{\Q}^{p,D_p}(Q)$ and $\mathscr{F}_{\Q}^{2p,D_p(2p)}(Q)$ to the behavior of $|\mathrm{Cl}_{K}[p]|$ as $K\in\mathscr{F}_{\Q}^2(Q)$ varies.  Using this relationship, Kl{\"u}ners \cite{Kluners} proved that $\mathscr{F}_{\Q}^{2p,D_p(2p)}(Q)\ll_{p,\epsilon} Q^{3/(2p)+\epsilon}$.  By combining the work of Kl{\"u}ners with the first line of \eqref{eqn:EPW1} when $\ell=p$, Cohen and Thorne \cite[Theorem 1.1]{CohenThorne} proved that $\mathscr{F}_{\Q}^{p,D_p}(Q)\ll_{p,\epsilon} Q^{3/(p-1)-1/(p(p-1))+\epsilon}$.  By combining the work of Kl{\"u}ners with the first line of \eqref{eqn:EPW3} when $r=1$ and $\ell=p$, Frei and Widmer \cite[Corollary 1.2]{FW} improved both bounds to
\begin{equation}
\label{eqn:FW_dihedral}
\mathscr{F}_{\Q}^{p,D_p}(Q)\ll_{p,\epsilon} Q^{\frac{3}{p-1}-\frac{2}{(p+2)(p-1)}+\epsilon},\qquad  \mathscr{F}_{\Q}^{2p,D_p(2p)}(Q)\ll_{p,\epsilon} Q^{\frac{3}{2p}-\frac{1}{p(p+2)}+\epsilon}.	
\end{equation}
Using \eqref{eqn:moment3} instead of \eqref{eqn:EPW3}, we immediately improve upon \eqref{eqn:FW_dihedral} as follows.
\begin{corollary}
If $p$ is an odd prime and $\epsilon>0$, then
\[
\mathscr{F}_{\Q}^{p,D_p}(Q)\ll_{p,\epsilon} Q^{\frac{3}{p-1}-\frac{2}{(p+1)(p-1)}+\epsilon},\qquad  \mathscr{F}_{\Q}^{2p,D_p(2p)}(Q)\ll_{p,\epsilon} Q^{\frac{3}{2p}-\frac{1}{p(p+1)}+\epsilon}.	
\]
\end{corollary}

\subsection*{Acknowledgements}
We thank Roger Heath-Brown and Lillian Pierce for their encouragement.  PK would like to thank Carlo Pagano and Martin Widmer for fruitful discussions. PK gratefully acknowledges the support of Dr. Max R\"ossler, the Walter Haefner Foundation and the ETH Z\"urich Foundation.  JT gratefully acknowledges the support of the Simons Foundation (MP-TSM-00002484).

\section{Notation, conventions, and overview}
\label{sec:notation}
Throughout this paper, we fix integers $d,\ell\geq 2$.  Let $\mathscr{F}_{k}^{d}$ be as in \eqref{eqn:families} and $\mathcal{S}\subseteq\mathscr{F}_{k}^{d}$.  For $K\in \mathcal{S}$, the set of places is denoted $\Omega_K$. For each $v\in\Omega_K$, let $K_v$ be the completion of $K$ with respect to $v$. We normalize the absolute value $|\cdot|_v$ such that it extends an absolute value of $\Q$ (either archimedean or $p$-adic). We also define $d_v := [K_v : \Q_v]$. The ring of integers of $K$ is $\OO_K$, the set of nonzero ideals of $\OO_K$ is $\mathcal{I}_K$, and the set of nonzero prime ideals of $\OO_K$ is $\mathcal{P}_K$.  Let $\Norm_{K/\Q}$ be the absolute norm of $K$, defined on $\mathfrak{a} \in I_K$ by $\Norm_{K/\Q}\mathfrak{a}=|\OO_K/\mathfrak{a}|$.  Let $\mathbb{N}$ denote the set of positive integers.

Given $\mathfrak{P}\in\mathcal{P}_K$ lying over $\mathfrak{p}\in\mathcal{P}_k$, we write $e(\mathfrak{P}) = e(\mathfrak{P}/\mathfrak{p})$ for the ramification index and $f(\mathfrak{P}) = f(\mathfrak{P}/\mathfrak{p})$ for the inertia degree of $\mathfrak{P}$ over $\mathfrak{p}$. We define
\begin{equation}
\label{eqn:rel_deg_1}
\mathcal{P}_K^{(1)} = \{\mathfrak{P}\in\mathcal{P}_K\colon e(\mathfrak{P})=f(\mathfrak{P})=1\},	
\end{equation}
the set of prime ideals of $K$ of relative degree 1 (over $k$), and the counting functions
\begin{equation}
\label{eqn:prime_counting}
\begin{aligned}
\pi_K(x)&:=|\{\mathfrak{P}\in \mathcal{P}_K\colon \Norm_{K/\Q}\mathfrak{P} \leq x\}|,\\
\pi_K^{(1)}(x)&:=|\{\mathfrak{P}\in \mathcal{P}_K^{(1)}\colon \Norm_{K/\Q}\mathfrak{P} \leq x\}|.
\end{aligned}
\end{equation}

For complex-valued functions $f$ and $g$, we write $f=O_{\nu}(g)$ or $f\ll_{\nu}g$ to denote that there exists an effectively computable constant $c>0$ (depending at most on $\nu$, $d$, $\ell$, and $k$) such that in a domain $U\subseteq\mathbb{C}$ that will be clear from context, we have that if $z\in U$, then $|f(z)|\leq c|g(z)|$. We write $f\asymp_{\nu}g$ to denote that $f\ll_{\nu}g$ and $g\ll_{\nu}f$.

\section{\texorpdfstring{A bound for $|\mathrm{Cl}_K[\ell]|$}{A bound for Cl(K)}}
\label{sec:main_lemma}
Let $\mathcal{S}\subseteq \mathscr{F}_{k}^{d}$ and $K\in \mathcal{S}$.  For $\alpha\in K$, let
\[
H_K(\alpha) = \prod_{v \in \Omega_K} \max(1, |\alpha|_v^{d_v})
\]
be the multiplicative Weil height. This height depends on $K$, whence the subscript.

\begin{mydef}
A lattice $\mathcal{L}$ is a discrete subgroup of $\mathbb{R}^n$. Write $\|\cdot\|$ for the Euclidean norm on $\mathbb{R}^n$. A basis $(\mathbf{b}_1, \dots, \mathbf{b}_m)$ of $\mathcal{L}$ is a Minkowski basis if for every $i$, $\mathbf{b}_i$ is vector with minimal $\|\cdot\|$ such that $(\mathbf{b}_1, \dots, \mathbf{b}_i)$ extends to a basis. This is equivalent to
\[
\|\mathbf{b}_i\| \leq \Big\|\sum_{j = 1}^m a_j \mathbf{b}_j\Big\|
\]
for every vector of integers $(a_1, \dots, a_m) \in \Z^m$ such that $\gcd(a_i, \dots, a_m)=1$.
\end{mydef}

\begin{lemma}
\label{lMinkowski}
Let $\mathcal{L} \subseteq \mathbb{R}^n$ be a lattice. Then there exists a Minkowski basis for $\mathcal{L}$.
\end{lemma}

\begin{proof}
This follows from a greedy algorithm.
\end{proof}

Given $K\in \mathcal{S}$, $\ell\geq 2$, and $Z>0$, we define $S_{\ell}(K,Z)$ to equal
\begin{equation}
\label{eqn:S_def}
\{\beta\in K\colon \textup{$H_K(\beta) \leq Z$, there exist distinct $\mathfrak{P}_1,\mathfrak{P}_2\in\mathcal{P}_K^{(1)}$ with $\beta\OO_K = (\mathfrak{P}_1\mathfrak{P}_2^{-1})^\ell$}\}.
\end{equation}
Our next result, a crucial bound for $|\mathrm{Cl}_K[\ell]|$ in terms of $S_{\ell}(K,Z)$ and $\pi_K^{(1)}(Z)$, gives a variant of \cite[Proposition 2.1]{HBP} for all extensions $K/k$ using ideas from \cite[Section 2]{FW}.

\begin{theorem}
\label{tFW}
Let $\epsilon>0$, $d \geq 2$, $\ell\geq 2$, and $Z>0$.  There exists an effectively computable constant $\Cl[abcon]{tFW}=\Cr{tFW}(d,k,\ell) > 0$ such that if $K\in \mathcal{S}$ and $\pi_K^{(1)}(Z)>0$, then
\[
|\mathrm{Cl}_{K}[\ell]| \ll_{\epsilon} D_K^{\frac{1}{2} + \epsilon}/\pi_K^{(1)}(Z) + D_K^{\frac{1}{2} + \epsilon} \cdot |S_\ell(K, \Cr{tFW}Z^{\ell})|/\pi_K^{(1)}(Z)^2.
\]
\end{theorem}

\begin{proof}
We write $R_K$ for the regulator of $K$ and define $A := \mathrm{Cl}_{K}/\mathrm{Cl}_{K}[\ell]$. The bound
\[
|\mathrm{Cl}_{K}[\ell]| \cdot |A| \cdot R_K = |\mathrm{Cl}_{K}| \cdot R_K \ll_{\epsilon} D_K^{\frac{1}{2} + \epsilon}
\]
follows from the bound $\mathrm{Res}_{s=1}\zeta_K(s)\ll (\log D_K)^{[K:\Q]-1}$ \cite[(2)]{Louboutin} and the analytic class number formula. It remains to show that there exists an effectively computable constant $\Cr{tFW}=\Cr{tFW}(d,k,\ell)>0$ such that
\begin{align}
\label{eClaim}
5|A| \cdot R_K \geq \left(\frac{1}{\pi_K^{(1)}(Z)} + \frac{|S_\ell(K, \Cr{tFW}Z^{\ell})|}{\pi_K^{(1)}(Z)^2}\right)^{-1}.
\end{align}

To prove \eqref{eClaim}, we let $v_1, \dots, v_r$ be the real places of $K$ and $v_{r + 1}, \dots, v_{r + s}$ be the complex places of $K$.  Let $d_i=1$ if $1\leq i\leq r$ and $d_i=2$ if $r+1\leq i\leq r+s$.  We have the classical homomorphism
\[
\varphi\colon K^{\times} \rightarrow \mathbb{R}^{r + s},\qquad \varphi(\alpha) = (d_i \log |\alpha|_{v_i})_{i=1}^{r+s}.
\]
Consider the subspace of $\mathbb{R}^{r + s}$ given by
\begin{align}
\label{eOrthogonal}
V_0 = \Big\{(x_1, \dots, x_{r + s})\in \R^{r+s}\colon \sum_{i = 1}^{r + s} d_i x_i = 0\Big\}.
\end{align}
By Dirichlet's unit theorem, $\mathcal{L} := \varphi(\mathcal{O}_K^{\times})$ is a full rank lattice inside $V_0$. By Lemma \ref{lMinkowski}, there exists a Minkowski basis $(\mathbf{u}_1, \dots, \mathbf{u}_{r + s - 1})$ of $\mathcal{L}$. If we define $\mathbf{d} := (d_i)_{1 \leq i \leq r + s}$, then $(\mathbf{u}_1, \dots, \mathbf{u}_{r + s - 1}, \mathbf{d})$ forms a full rank lattice inside $\mathbb{R}^{r + s}$ by \eqref{eOrthogonal}. Let
\[
\mathcal{F} = \Big\{\sum_{i = 1}^{r + s - 1} x_i \mathbf{u}_i : 0 \leq x_i < 1\Big\}
\]
be a fundamental domain of $\mathcal{L}$.  For all $\alpha\in K$, there exists $g_{\alpha}\in\OO_K^{\times}$ (unique up to multiplication by a root of unity), $\mathbf{v}_{\alpha}\in\mathcal{F}$, and $y_{\alpha}\in\R$ such that $\varphi(g_{\alpha}\alpha) = \mathbf{v}_{\alpha}+y_{\alpha}\cdot \mathbf{d}$.

We claim that 
\begin{align}
\label{eLowerui}
\|\mathbf{u}_i\| \gg 1. 
\end{align}
Indeed, if $u_i\in \varphi^{-1}(\mathbf{u}_i)$, in which case $u_i\in\OO_K^{\times}$ and $\mathbf{u}_i = (d_j \log|u_i|_{v_j})_{1\leq j\leq r+s}$, then the inequality of arithmetic and geometric means yields
\[
\|\mathbf{u}_i\| = \Big(\sum_{j = 1}^{r + s} (d_j \log |u_i|_{v_j})^2\Big)^{1/2} \geq \frac{\sum_{j = 1}^{r + s} d_j \max(0, \log |u_i|_{v_j})}{\sqrt{r + s}} \geq \frac{\sum_{j = 1}^{r + s} d_j \max(0, \log |u_i|_{v_j})}{\sqrt{d[k : \Q]}}.
\]
Since $u_i\in\OO_K^{\times}$, the right hand side of the above inequality equals $(\log H_K(u_i))/\sqrt{d[k : \Q]}$. Therefore, the claim \eqref{eLowerui} follows from Northcott's theorem.

We will apply Minkowski's second theorem to $V_0$. Let $\pi: V_0 \rightarrow \mathbb{R}^{r + s - 1}$ be the isomorphism that drops the last coordinate, and let $\iota$ be the compositional inverse of $\pi$. We endow $V_0$ with the pushforward measure, via $\iota$, of the Lebesgue measure on $\mathbb{R}^{r + s - 1}$. Since $(\mathbf{u}_1, \dots, \mathbf{u}_{r + s - 1})$ is a Minkowski basis, Minkowski's second theorem applied to $V_0$ yields
\begin{equation}
\label{eqn:upper_prod}
\prod_{i =1}^{r + s - 1} \|\mathbf{u}_i\| \ll \Vol(\mathcal{F}).	
\end{equation}
Together, \eqref{eLowerui} and \eqref{eqn:upper_prod} yield $\|\mathbf{u}_i\| \ll \Vol(\mathcal{F})$. Since the volume of $\pi(\Vol(\mathcal{F}))$ equals $R_K$ by definition, we obtain $\|\mathbf{u}_i\| \ll R_K$.  Since $R_K \geq 0.205$ by \cite{Friedman}, we have that $\lceil R_K\rceil<5R_K$.  Thus, there exists an effectively computable constant $\Cl[abcon]{cell}=\Cr{cell}(d,k,\ell)>0$ and an integer $n\in[1,5R_K)$ such that $\mathcal{F}$ partitions into $n$ cells $\mathcal{F}_1,\ldots,\mathcal{F}_{n}$, each with diameter at most $\Cr{cell}$.

Fix a full set of integral representatives $\mathfrak{b}_1,\ldots,\mathfrak{b}_{|\mathrm{Cl}_K|}$ of the ideal class group $\mathrm{Cl}_K$.  For each $\mathfrak{P}\in\mathcal{P}_K^{(1)}$, there exists a unique integer $j_{\mathfrak{P}}\in\{1,\ldots,|\mathrm{Cl}_K|\}$ such that $\mathfrak{b}_{j_{\mathfrak{P}}}\mathfrak{P}^{\ell}$ is a principal ideal.  Given a prime ideal $\mathfrak{P} \in \mathcal{P}_{K}^{(1)}$, let $a=a_{\mathfrak{P}}$ be the image of $\mathfrak{P}$ in $A$ and $i$ be the least integer such that a generator $\alpha$ of $\mathfrak{b}_{j_{\mathfrak{P}}} \mathfrak{P}^\ell$ satisfies $\mathbf{v}_\alpha \in \mathcal{F}_i$.  We define the map
\[
f\colon \{\mathfrak{P}\in\mathcal{P}_K^{(1)}\colon \mathrm{N}_{K/\Q}\mathfrak{P}\leq Z\} \to A \times \{1, \dots, n\},\qquad f(\mathfrak{P})=(a,i).
\]

The Cauchy--Schwarz inequality yields
\[
\pi_K^{(1)}(Z) = \sum_{(a, i)} |f^{-1}(a, i)| \leq \Big(\sum_{\substack{(a, i)\colon f^{-1}(a, i) \neq \varnothing}} 1\Big)^{1/2} \Big(\sum_{\substack{(a, i)\colon f^{-1}(a, i) \neq \varnothing}} |f^{-1}(a, i)|^2\Big)^{1/2},
\]
hence
\[
5|A| \cdot R_K \geq |A| \cdot \lceil R_K\rceil \geq \sum_{\substack{(a, i)\colon f^{-1}(a, i) \neq \varnothing}} 1 \geq \frac{\pi_K^{(1)}(Z)^2}{\sum_{\substack{(a, i)\colon f^{-1}(a, i) \neq \varnothing}} |f^{-1}(a, i)|^2}.
\]
It remains to produce constant $\Cr{tFW}=\Cr{tFW}(d,k,\ell)>0$ such that
\begin{equation}
\label{eqn:final_sum}
\sum_{\substack{(a, i)\colon f^{-1}(a, i) \neq \varnothing}} |f^{-1}(a, i)|^2 \leq \pi_K^{(1)}(Z) + |S_\ell(K, \Cr{tFW}Z^{\ell})|.	
\end{equation}

In the Cartesian product $f^{-1}(a, i) \times f^{-1}(a, i)$, whose cardinality is $|f^{-1}(a, i)|^2$, there are the diagonal elements $(\mathfrak{P}, \mathfrak{P})$ and off-diagonal elements $(\mathfrak{P}_1, \mathfrak{P}_2)$ with $\mathfrak{P}_1 \neq \mathfrak{P}_2$.  We treat these contributions differently.   The diagonal elements contribute $\pi_K^{(1)}(Z)$.  For the off-diagonal contribution, take distinct $\mathfrak{P}_1, \mathfrak{P}_2\in\mathcal{P}_K^{(1)}$. This implies that $\mathfrak{P}_1/\mathfrak{P}_2$ is $\ell$-torsion, and thus $\mathfrak{P}_1^\ell$ and $\mathfrak{P}_2^\ell$ have the same image in $\mathrm{Cl}_K$. Therefore, there exists one class $\mathfrak{b}_j$ and elements $\alpha_1, \alpha_2$ such that $\mathfrak{b}_j \mathfrak{P}_1^\ell = (\alpha_1)$, $\mathfrak{b}_j \mathfrak{P}_2^\ell = (\alpha_2)$, and $\varphi(g_{\alpha_1}\alpha_1), \varphi(g_{\alpha_2}\alpha_2) \in \mathcal{F}_i$. 

By interchanging $\alpha_1$ and $\alpha_2$ if necessary, we may assume that $y_{\alpha_1} \leq y_{\alpha_2}$.  Since $\mathbf{v}_{\alpha_1}, \mathbf{v}_{\alpha_2}\in\mathcal{F}_i$, there exists a constant $\Cl[abcon]{cell2}=\Cr{cell2}(d,k)>0$ such that if $v\mid\infty$, then 
\[
\log |g_{\alpha_1} \alpha_1|_v - \log |g_{\alpha_2} \alpha_2|_v\leq \Cr{cell2}. 
\]
In other words, we have that $|\frac{g_{\alpha_1} \alpha_1}{g_{\alpha_2} \alpha_2}|_v \leq e^{\Cr{cell2} + y_{\alpha_1} - y_{\alpha_2}} \leq e^{\Cr{cell2}}$.  Therefore, as desired, there exists a constant $\Cr{tFW} = \Cr{tFW}(d,k,\ell, \epsilon) > 0$ such that $H_K(\frac{g_{\alpha_1} \alpha_1}{g_{\alpha_2} \alpha_2}) \leq \Cr{tFW} N(\mathfrak{P}_2)^\ell \leq \Cr{tFW}Z^\ell$.
\end{proof}

Theorem \ref{thm:main} follows from showing that $S_{\ell}(K,Z)$ is small and $\pi_K^{(1)}(Z)$ is large on average over $K\in\mathcal{S}$.  These are proved in Lemma \ref{letaBound} and Corollary \ref{cor:LOTZ}, respectively.

\section{Estimating $S_{\ell}(K,Z)$ on average}

Let $\mathcal{S}\subseteq\mathscr{F}_k^{d}$.  For each $K\in\mathcal{S}$, consider $S_{\ell}(K,Z)$ from \eqref{eqn:S_def}.  With small modifications to the case of $k=\Q$ considered in \cite{FW}, the ideas in the proof of \cite[Lemma 4.2]{FW} lead to
\begin{equation}
\label{eqn:FW_eta}
\sum_{K\in \mathcal{S}} |S_\ell(K, Z)| \ll_{\epsilon} Z^{d[k:\Q] - 1 + 2/\ell + \epsilon}.
\end{equation}
The main result of this section, Lemma \ref{letaBound}, refines \eqref{eqn:FW_eta}.

\begin{lemma}
\label{letaBound}
If $Z\geq 2$ and $\epsilon>0$, then $\sum_{K\in \mathcal{S}} |S_\ell(K, Z)| \ll_{\epsilon} Z^{d - 1 + 2/\ell + \epsilon}$.
\end{lemma}

\subsection{Prerequisites}
For this subsection, fix a number field $F$ and a full set $\mathfrak{C}_1, \dots, \mathfrak{C}_{|\mathrm{Cl}_F|}\in\mathcal{I}_F$ of integral representatives for the ideal class group $\text{Cl}_F$.

\begin{lemma}
\label{lSmallGen}
If $\mathfrak{a}$$\,\in\,$$\mathcal{I}_F$, then there exists $\alpha$$\,\in\,$$F^{\times}$ such that $H_F(\alpha)\ll\mathrm{N}_{F/\Q}(\alpha)$ and $\mathfrak{a} = (\alpha)$.
\end{lemma}
\begin{proof}
This is a special case of \cite[Proposition 4.3.12]{MR3524535}.
\end{proof}
\begin{lemma}
\label{lFraction}
There exists a constant $\Cl[abcon]{lFrac}=\Cr{lFrac}(F)>0$ such that for all $\alpha \in F^{\times}$, there exists $(t,t')\in \OO_F\times \OO_F - \{0\}$ such that if $\alpha = t/t'$, then $H_F(t), H_F(t') \leq \Cr{lFrac} H_F(\alpha)$.  Also, there exists $i\in\{1,\ldots,|\mathrm{Cl}_F|\}$ such that $\gcd((t), (t')) = \mathfrak{C}_i$.
\end{lemma}

\begin{proof}
Write $(\alpha) = \prod_{\mathfrak{p}\in\mathcal{P}_F} \mathfrak{p}^{e_\mathfrak{p}}$, and define $I = \prod_{\substack{e_{\mathfrak{p}} < 0}} \mathfrak{p}^{-e_\mathfrak{p}}$.  Let $\mathfrak{C}_i$ be the unique representative such that $I^{-1} \sim \mathfrak{C}_i$ in $\text{Cl}_F$. Therefore, $I \mathfrak{C}_i$ is a principal ideal of $\OO_F$.  By Lemma \ref{lSmallGen}, there exists a generator $t'$ of $I \mathfrak{C}_i$ satisfying $H_F(t') \ll \mathrm{N}_{F/\Q}(I \mathfrak{C}_i) \ll \mathrm{N}_{F/\Q}(I)$.

Finally, we choose $t = \alpha t'$. Then we have $t \in \OO_F$ and $\gcd((t), (t')) = \mathfrak{C}_i$ by construction. Furthermore, since $t$ and $t'$ are integral, we see that
\begin{multline*}
H_F(t) = \prod_{v \mid \infty} \max(1, |\alpha t'|_v^{d_v}) \leq \prod_{v \mid \infty} \max(1, |\alpha|_v^{d_v}) \prod_{v \mid \infty} \max(1, |t'|_v^{d_v}) \\
= H_F(\alpha) \frac{\prod_{v \mid \infty} \max(1, |t'|_v^{d_v})}{\prod_{v \text{ finite}} \max(1, |\alpha|_v^{d_v})} = H_F(\alpha) \frac{H_F(t')}{\mathrm{N}_{F/\Q}(I)} \ll H_F(\alpha).\qedhere
\end{multline*}
\end{proof}

\begin{lemma}
\label{lUnitCount}
If $\epsilon > 0$, then $|\{\textup{$u \in \OO_F^{\times}\colon$ if $v\mid\infty$, then $|u|_v \leq X$}\}| \ll_{\epsilon} X^\epsilon$.
\end{lemma}

\begin{proof}
This follows from the main theorem of \cite{EverestLoxton}.
\end{proof}

\begin{lemma}
\label{lGeneratorCount}
If $\alpha \in F^{\times}$, $H_F(\alpha) \leq X$, and $\epsilon>0$, then $|\{u \in \OO_F^{\times} : H_F(\alpha u) \leq X\}| \ll_\epsilon X^\epsilon$.
\end{lemma}

\begin{proof}
We have that $|\alpha|_v^{d_v} \geq \frac{1}{X}$ for all $v \mid \infty$, since otherwise $H_F(\alpha) = H_F(\alpha^{-1}) > X$. Therefore the condition $H_F(\alpha u) \leq X$ forces $|u|_v^{d_v} \leq X^2$ for all $v \mid \infty$. Since $d_v \geq 1$, the desired result follows from Lemma \ref{lUnitCount}.
\end{proof}

\begin{lemma}
\label{lWidmer}
If $\epsilon > 0$, then $|\{\alpha \in \OO_F : H_F(\alpha) \leq X\}| \ll_\epsilon X^{1 + \epsilon}$.
\end{lemma}

\begin{proof}
This follows from work of Widmer \cite[Theorem 1.1 with $e = n = 1$]{Widmer_points}. Note that Widmer works with the absolute height, while we work with the relative height.
\end{proof}

\subsection{The results}
Recall that $K\in \mathcal{S}$, as in Section \ref{sec:main_lemma}. We estimate the average of $|S_\ell(K, Z)|$ using the following lemma. Given $\alpha \in K$, we write $f_\alpha(X) \in k[X]$ for the unique monic minimal polynomial of $\alpha$ over $k$. In this subsection, we prove variants of \cite[Lemmata 4.1 and 4.2]{FW} for general extensions $K/k$.

\begin{lemma}
\label{lMinimal}
Let $K\in \mathcal{S}$ and $\alpha \in K$. If there exist distinct $\mathfrak{P}_1,\mathfrak{P}_2\in\mathcal{P}_K^{(1)}$ such that $\alpha \mathcal{O}_K = (\mathfrak{P}_1 \mathfrak{P}_2^{-1})^\ell$, then $K = k(\alpha)$, and there exist unique $a_1, \dots, a_d \in k$ such that
\begin{equation}
\label{eqn:min_poly}
f_\alpha(X) = X^d + a_1 X^{d - 1} + \dots + a_d.
	\end{equation}
The elements $a_i$ have the following properties.
\begin{itemize}
\item If $1\leq i\leq d$, then $H_k(a_i) \ll H_K(\alpha)$.
\item The fractional ideal $a_d\OO_k$ is an $\ell$-th power.
\item If $v \in \Omega_k$ is a finite place such that $v(a_i) < 0$, then $v$ lies below $\mathfrak{P}_2$.
\end{itemize}
\end{lemma}

\begin{proof}
First, we assert that $\alpha$ has degree $d$ over $k$, which immediately implies that $K = k(\alpha)$. Indeed, if the claim is false, then there exists a proper subfield $K'\subseteq K$ such that $\alpha\in K'$. Writing $\mathfrak{q}$ for the prime ideal of $K'$ below $\mathfrak{P}_1$, we see that $\mathfrak{q}$ divides $\alpha \mathcal{O}_{K'}$. But since $e(\mathfrak{P}_1) = f(\mathfrak{P}_1) = 1$, there exists another prime ideal $\mathfrak{P}$ of $K$ above $\mathfrak{q}$, which must then divide $\alpha \mathcal{O}_K$ as well. This is a contradiction; therefore, our assertion is true.

Recall that $f_\alpha(X) \in k[X]$ is the minimal polynomial of $\alpha$ over $k$. Since $\alpha$ has degree $d$ over $k$, there exist unique $a_1, \dots, a_d \in k$ such that $f_{\alpha}(X)$ is of the form \eqref{eqn:min_poly}.  Writing $\alpha = \alpha^{(1)}, \dots, \alpha^{(d)} \in \overline{\Q}$ for the conjugates of $\alpha$ over $k$, we have
\[
f_\alpha(X) = \prod_{i = 1}^d (X - \alpha^{(i)}).
\]
Vieta's formulas relating the numbers $\alpha^{(1)}, \ldots, \alpha^{(d)}$ to the coefficients $a_1, \ldots, a_d$ imply that $H_k(a_i) = H_K(a_i)^{1/d} \ll H_K(\alpha)$. Furthermore, since $a_d = \prod_{i = 1}^d -\alpha^{(i)} = (-1)^d \mathrm{N}_{K/k}(\alpha)$, we see that $a_d\OO_k$ is the $\ell$-th power of an ideal. Finally, suppose that $v(a_i) < 0$. Using Vieta's formulas once more, we deduce that there exists some place $w$ of $\widetilde{K}$ above $v$ and some $j$ such that $w(\alpha^{(j)}) < 0$. Therefore there exists some place $\widetilde{w}$ of $\widetilde{K}$ above $v$ such that $\widetilde{w}(\alpha) < 0$. This implies that $\widetilde{w}$ lies above $\mathfrak{P}_2$, finishing the proof.
\end{proof}

\begin{proof}[Proof of Lemma \ref{letaBound}]
Write $T$ for the set of all $\alpha \in \overline{\Q}$ such that there exist distinct prime ideals $\mathfrak{P}_1, \mathfrak{P}_2\in\mathcal{P}_K^{(1)}$ such that $k(\alpha) \in \mathcal{S}$ and $\alpha \mathcal{O}_{k(\alpha)} = (\mathfrak{P}_1 \mathfrak{P}_2^{-1})^\ell$. We define
\[
N_H(Z) := \{\alpha \in T : H_{k(\alpha)}(\alpha) \leq Z\}.
\]
By Lemma \ref{lMinimal}, the map $N_H(Z) \rightarrow \bigsqcup_{K\in \mathcal{S}} S_\ell(K, Z)$ that sends $\alpha$ to $(k(\alpha), \alpha)$ is surjective.  If $\alpha \in T$, then Lemma \ref{lMinimal} shows that its minimal polynomial $f_\alpha$ is of the shape $f_\alpha(X) = X^d + a_1 X^{d - 1} + \dots + a_d$ with $H_k(a_i) \ll H_{k(\alpha)}(\alpha)$, $a_d\OO_k$ an $\ell$-th power and $v(a_i) < 0$ implies that $\mathfrak{P}_2$ lies above $v$.  In order to estimate the number of such polynomials $f_\alpha$, we use Lemma \ref{lFraction} (with $F=k$) to write each $a_i$ as
\[
a_i = t_i/t_i', \quad H_k(t_i), H_k(t_i') \ll H_k(a_i) \ll H_{k(\alpha)}(\alpha) \leq Z, \quad \gcd((t_i), (t_i')) = \mathfrak{C}_{i}.
\]
By Lemma \ref{lWidmer} (with $F=k$), there are at most $\ll_{\epsilon} Z^{d - 1 + \epsilon}$ possibilities for $t_1, \dots, t_{d - 1}$. Furthermore, each of $t_1', \dots, t_{d - 1}'$ divides $\Norm_{k(\alpha)/k}(\mathfrak{P}_2)^\ell \prod_{i = 1}^h \mathfrak{C}_i$. Note that $\Norm_{k(\alpha)/\Q}(\mathfrak{P}_2)^\ell \leq H_{k(\alpha)}(\alpha) \leq Z$, thus $\Norm_{k(\alpha)/\Q}(\mathfrak{P}_2) \leq Z^{1/\ell}$. The number of ideals of $k$ with norm up to $Z^{1/\ell}$ is bounded by $\ll Z^{1/\ell}$. Since $t_1', \dots, t_{d - 1}'$ all divide $\Norm_{k(\alpha)/k}(\mathfrak{P}_2)^\ell \prod_{i = 1}^h \mathfrak{C}_i$, Lemma \ref{lGeneratorCount} gives $\ll_\epsilon Z^{1/\ell + \epsilon}$ possibilities for $t_1', \dots, t_{d - 1}'$. Arguing similarly for $a_d$ produces another $\ll_\epsilon Z^{1/\ell + \epsilon}$ possibilities, completing the proof of the lemma.
\end{proof}

\section{Proof of Theorem \ref{thm:main}}
\label{sec:proofs}
Throughout this section, we fix integers $n,\ell\geq 2$, a prime $p$, and a number field $k$.  Recall \eqref{eqn:families} and \eqref{eqn:prime_counting}.  Let the ordered pair $(\mathscr{F}_k,d)$ equal $(\mathscr{F}_k^{n,S_n},n)$ or $(\mathscr{F}_k^{p},p)$, and let $\mathcal{S}\subseteq\mathscr{F}_k$.  We may assume that $\alpha_{\mathcal{S}}>0$; otherwise, Theorem \ref{thm:main} is trivial.  Let $Q\geq 1$, let $0<\epsilon<\min\{\frac{1}{2},\alpha_{\mathcal{S}}\}$ be an arbitrarily small quantity, and define
\[
R = \ell(d-1)+1.
\]

It remains to bound for $\pi_K^{(1)}(Z)$ in Theorem \ref{tFW} from below.  To obtain such a lower bound on average, we use the work of Lemke Oliver, Thorner, and Zaman \cite{LOTZ}.

\begin{theorem}
\label{thm:LOTZ}
Recall \eqref{eqn:prime_counting}.  There exists an effectively computable constant $\Cl[abcon]{LOTZ_n}=\Cr{LOTZ_n}(k,n)>0$ such that for all $K\in \mathscr{F}_k^{n,S_n}(Q)$ with $O_{\epsilon}(Q^{\epsilon})$ exceptions, there holds
\[
\pi_K(x) \geq 2\Cr{LOTZ_n} x/\log x,\qquad x\geq (\log D_K)^{42(n!)^2/\epsilon}.
\]
Also, there exists an effectively computable constant $\Cl[abcon]{LOTZ_p}=\Cr{LOTZ_p}(k,p)>0$ such that if $G$ is a transitive subgroup of $S_p$, then for all $K\in\mathscr{F}_k^{p,G}(Q)$ with $O_{\epsilon}(Q^{\epsilon})$ exceptions, there holds
\[
\pi_K(x) \geq 2\Cr{LOTZ_p} x/\log x,\qquad x\geq (\log D_K)^{42(p!)^2/\epsilon}.
\]
\end{theorem}

\begin{proof}
This is established during the proof of \cite[Theorem 2.4]{LOTZ}.
\end{proof}

\begin{corollary}
\label{cor:LOTZ}
There exists an effectively computable constant $\Cl[abcon]{LOTZ_d}=\Cr{LOTZ_d}(d,k)>0$ and a subset $\mathscr{E}_k(Q)\subseteq\mathscr{F}_k(Q)$ satisfying $|\mathscr{E}_k(Q)|\ll_{\epsilon}Q^{\epsilon}$ such that if $K\in \mathcal{S}(Q)-\mathscr{E}_k(Q)$, then
\[
\pi_K^{(1)}(x) \geq \Cr{LOTZ_d} x/\log x,\qquad x\geq (\log D_K)^{42(d!)^2/\epsilon}.
\]
\end{corollary}

\begin{proof}
First, let $(\mathscr{F}_k,d)=(\mathscr{F}_k^{n,S_n},n)$.  If $\mathfrak{P}\in\mathcal{P}_K-\mathcal{P}_K^{(1)}$, then the relative degree (over $k$) of $\mathfrak{P}$ is at least 2.  Consequently, the absolute degree (over $\Q$) of $\mathfrak{P}$ is at least 2.  There are $O(\sqrt{x})$ such prime ideals $\mathfrak{P}$ such that $\mathrm{N}_{K/\Q}\mathfrak{P}\leq x$.  The corollary now follows from Theorem \ref{thm:LOTZ}.  If $(\mathscr{F}_k,d)=(\mathscr{F}_k^{p},p)$, then we apply this reasoning to $\mathscr{F}_k^{p,G}$ for each transitive subgroup of $S_p$.  Since there are $O_p(1)$ such subgroups, the theorem follows.
\end{proof}

In a manner inspired by \cite[Section 6]{HBP}, we use Corollary \ref{cor:LOTZ} to bound the frequency with which $|\mathrm{Cl}_K[\ell]|$ can be large. Define $A_\ell(Q; H) := \{K\in \mathcal{S}\colon |\mathrm{Cl}_K[\ell]| \geq H\}$.

\begin{proposition}
\label{pAHQ}
If $H,Q\geq 1$, then $|A_\ell(Q; H)| \ll_\epsilon \min\{|\mathcal{S}(Q)|,Q^{R/2 + \epsilon}/H^{R}\}$.
\end{proposition}

\begin{proof}
The bound $|A_{\ell}(Q; H)|\leq |\mathcal{S}(Q)|$ is trivial.  For the remaining part, it follows from \eqref{eqn:trivial_minkowski} that $|A_\ell(Q; H)|=0$ unless $H\ll Q^{1/2}(\log Q)^{d[k:\Q]-1}$.  We begin with the estimate
\[
|A_{\ell}(Q;H)|\leq\frac{1}{H}\sum_{K\in A_{\ell}(Q;H)}|\mathrm{Cl}_K[\ell]|.
\]
Corollary \ref{cor:LOTZ} and \eqref{eqn:trivial_minkowski} imply that
\[
\frac{1}{H}\sum_{K\in A_{\ell}(Q;H)\cap \mathscr{E}_k(Q)}|\mathrm{Cl}_K[\ell]|\ll \frac{Q^{\frac{1}{2}}(\log Q)^{d[k:\Q]-1}}{H}.
\]
For the other $K$, it suffices to assume that $Q$ is sufficiently large in terms of $d$, $k$, $\ell$, and $\epsilon$. In particular, there exists a constant $\Cl[abcon]{beat_lower}=\Cr{beat_lower}(d,k,\ell,\epsilon)\geq 1$ such that if $Q\geq\Cr{beat_lower}$, then
\[
Z := Q^{\frac{1}{2}+\frac{\epsilon}{R}}/H>(\log Q)^{42(d!)^2/\epsilon}.
\]
We now apply Theorem \ref{tFW}, Lemma \ref{letaBound}, and Corollary \ref{cor:LOTZ} with this choice of $Z$:
\begin{align*}
\frac{1}{H}\sum_{K\in A_{\ell}(Q;H)-\mathscr{E}_k(Q)}|\mathrm{Cl}_K[\ell]|&\ll_{\epsilon}\frac{1}{H}\sum_{K\in A_{\ell}(Q;H)-\mathscr{E}_k(Q)}\Big(\frac{D_K^{\frac{1}{2}+\frac{\epsilon}{4R}}}{\pi_K^{(1)}(Z)}+\frac{D_K^{\frac{1}{2}+\frac{\epsilon}{4R}}|S_{\ell}(K,\Cr{tFW}Z^{\ell})|}{\pi_K^{(1)}(Z)^2}\Big)\\
&\ll_{\epsilon}  \frac{Q^{\frac{1}{2}+\frac{\epsilon}{4R}}(\log Z)^2}{H}\sum_{K\in A_{\ell}(Q;H)-\mathscr{E}_k(Q)}\Big(\frac{1}{Z}+\frac{|S_{\ell}(K,\Cr{tFW}Z^{\ell})|}{Z^2}\Big)\\
&\ll_{\epsilon}  \frac{Q^{\frac{1}{2}+\frac{\epsilon}{2R}}}{H}\Big(\frac{|A_{\ell}(Q;H)|}{Z}+Z^{R-1}\Big).
\end{align*}
We conclude that there exists a constant $\Cl[abcon]{5.3upper}=\Cr{5.3upper}(d,k,\ell,\epsilon)\geq 1$ such that
\[
|A_{\ell}(Q;H)|\leq \Cr{5.3upper}\frac{Q^{\frac{1}{2}+\frac{\epsilon}{2R}}}{H}\Big(\frac{|A_{\ell}(Q;H)|}{Z}+Z^{R-1}\Big).
\]
We may assume that $\Cr{beat_lower}\geq (\Cr{5.3upper}+1)^{2R/\epsilon}$.  Now, by our choice of $Z$, the lemma follows.
\end{proof}

\begin{proof}[Proof of Theorem \ref{thm:main}]
Define $B_\ell(Q;H) := \{K\in \mathcal{S}\colon H \leq |\mathrm{Cl}_K[\ell]| < eH\}$.  By \eqref{eqn:trivial_minkowski} and Proposition \ref{pAHQ}, there exists a constant $\Cl[abcon]{dyadic}=\Cr{dyadic}(d,k)>0$ such that if $J=\log(\Cr{dyadic} Q^{\frac{1}{2}}(\log Q)^{d[k:\Q]-1})$, then
\begin{equation}
\label{eFinalBound}
\sum_{K\in \mathcal{S}} |\mathrm{Cl}_{K}[\ell]|^r \leq \sum_{0\leq j\leq J} |B_\ell(Q;e^j)| \cdot (e^{j+1})^r \ll_{r,\epsilon}\sum_{0\leq j\leq J} \min\Big(\frac{Q^{\frac{R}{2} + \frac{\epsilon}{2r}}}{e^{jR}},|\mathcal{S}(Q)|\Big) e^{jr}.
\end{equation}
Define
\[
j_0 = \frac{1}{R}\log\Big(\frac{Q^{\frac{R}{2}+\frac{\epsilon}{2r}}}{|\mathcal{S}(Q)|}\Big).
\]
By \eqref{eqn:EV_count} and our range of $\epsilon$, there exists a constant $\Cl[abcon]{j0_lower}=\Cr{j0_lower}(\alpha_{\mathcal{S}},d,k,\ell,\epsilon)>0$ such that if $Q\geq \Cr{j0_lower}$, then $1<j_0<J$.  Therefore, \eqref{eFinalBound} is
\begin{align*}
\ll_{r,\epsilon}|\mathcal{S}(Q)|\sum_{0\leq j\leq j_0}e^{jr}+Q^{\frac{R}{2}+\frac{\epsilon}{2r}}\sum_{j_0<j\leq J}e^{j(r-R)}\ll_{r,\epsilon}|\mathcal{S}(Q)|e^{j_0r}+Q^{\frac{R}{2}+\frac{\epsilon}{2r}}\sum_{j_0<j\leq J}e^{j(r-R)}.
\end{align*}
Theorem \ref{thm:main} now follows from the bounds $|\mathcal{S}(Q)|e^{j_0 r}\ll_{r,\epsilon} |\mathcal{S}(Q)|^{1-\frac{r}{R}}Q^{\frac{r}{2}+\frac{\epsilon}{2R}}$ and
\[
Q^{\frac{R}{2}+\frac{\epsilon}{2r}}\sum_{j_0<j\leq J}e^{j(r-R)}\ll_{r,\epsilon} \begin{cases}
	|\mathcal{S}(Q)|^{1-\frac{r}{R}}Q^{\frac{r}{2}+\epsilon}&\mbox{if $r< R$,}\\
	Q^{\frac{r}{2}+\epsilon}&\mbox{if $r\geq R$.}
\end{cases}\qedhere
\]
\end{proof}

\bibliographystyle{abbrv}
\bibliography{Torsion}
\end{document}